\documentclass[11pt,a4paper,reqno]{amsart}
\usepackage{hyperref}
\usepackage{amssymb, amsmath, amsthm}
\usepackage{bbm}
\usepackage{color}

\numberwithin{equation}{section}

\def\R{\mathbb{R}}
\def\N{\mathbb{N}}

\def\Z{\mathbb{Z}}

\def\mc{\mathcal}

\def\lesim{\lesssim}
\def\beq{\begin{equation}}
\def\endeq{\end{equation}}
\def\supp{\text{supp}}
\def\dim{\text{dim}}

\title{Polynomial Roth theorems on sets of fractional dimensions}

\author[R. Fraser]{Robert Fraser} 
\address{Robert Fraser: School of Mathematics, the University of Edinburgh, Edinburgh, UK}
\email{Robert.Fraser@ed.ac.uk}

\author[S. Guo]{Shaoming Guo}
\address{Shaoming Guo: Department of Mathematics, University of Wisconsin Madison, USA}
\email{shaomingguo2018@gmail.com}

\author[M. Pramanik]{Malabika Pramanik}
\address{Department of Mathematics, University of British Columbia, Vancouver, Canada}
\email{malabika@math.ubc.ca}

\date{\today}

\theoremstyle{plain}
\newtheorem{thm}{Theorem}[section]
\newtheorem{prop}[thm]{Proposition}

\newtheorem{lem}[thm]{Lemma}

\newtheorem{claim}[thm]{Claim}

\newtheorem*{conj*}{Conjecture}
\newtheorem*{openproblem*}{Open Problem}

\theoremstyle{definition}

\theoremstyle{remark}

\begin{document}
\maketitle

\begin{abstract}
Let $E\subset \R$ be a closed set of Hausdorff dimension $\alpha\in (0, 1)$. Let $P: \R\to \R$ be a polynomial without a constant term whose degree is bigger than one. We prove that if $E$ supports a probability measure satisfying certain dimension condition and Fourier decay condition, then $E$ contains three points $x, x+t, x+P(t)$ for some $t>0$. Our result extends the one of \L aba and the third author \cite{LP09} to the polynomial setting, under the same assumption. It also gives an affirmative answer to a question in Henriot, \L aba and the third author \cite{HLP15}.
\end{abstract}

\section{\bf Statement of results}

Let $P: \R\to \R$ be a polynomial without a constant term whose degree is greater than one. We prove 
\begin{thm}\label{180817thm1.1}
There exists $s_0>0$, depending only on the polynomial $P$, such that the following statement holds: For every $\beta\in (1-s_0, 1)$ and positive real numbers $C_1, C_2$ and $B$, there exists $\alpha_0\in (0, 1)$ such that, if $E\subset [0, 1]$ is a closed set that supports a probability measure $\mu$ satisfying
\beq\label{180817e1.1}
\begin{split}
& (A) \hspace{1cm} \mu([x, x+\epsilon])\le C_1 \epsilon^{\alpha} \text{ for all } 0<\epsilon\le 1,\\
& (B) \hspace{1cm} |\hat{\mu}(k)|\le C_2 (1-\alpha)^{-B} |k|^{-\frac{\beta}{2}} \text{ for all } k\in \Z, k\neq 0,
\end{split}
\endeq
with $1>\alpha>\alpha_0$, then $E$ contains three points 
\beq
x, x+t, x+P(t),
\endeq
for some $t>0$. 
\end{thm}

Later we will see that the value of $t$ in the above theorem is extremely small. When $P(t)$ does not contain any constant term, the value of $P(t)$ will also be small. In this sense, we say that our problem is local. The problem of dealing with a polynomial $P(t)$ containing a constant term is more global, and is not covered here. 

If one takes $P(t)=2t$, then it is a result due to \L aba and the third author \cite{LP09} that every $E\subset [0, 1]$ satisfying the same assumptions as in Theorem \ref{180817thm1.1} contains a non-trivial three-term arithmetic progression, that is $(x, x+t, x+2t)$ for some $t>0$. Moreover, $s_0$ was given explicitly there, which could be $1/3$. Indeed, it was also in \cite{LP09} that the assumptions $(A)$ and $(B)$ first appeared. These assumptions turn out to be very natural in the context of Salem sets. Let us briefly recall the discussion in \cite{LP09}. 

For a set $E\subset [0, 1]$, we let $\dim_{\mc{H}}(E)$ denote the Hausdorff dimension of the set $E$. We define the \emph{Fourier dimension} of $E$ to be the supremum over all $\beta\in [0, 1]$ such that there exists a probability measure $\mu$ supported on $E$ satisfying 
\beq
|\widehat{\mu}(\xi)|\le C(1+|\xi|)^{-\beta/2} \text{ for every } \xi\in \R \text{ and some } C>0.
\endeq
We will use $\dim_{\mc{F}}(E)$ to denote the Fourier dimension of $E$. Regarding the connection between Hausdorff dimensions and Fourier dimensions, it is known that 
\beq\label{181207e1.3}
\dim_{\mc{F}}(E)\le \dim_{\mc{H}}(E) \text{ for every } E.
\endeq
Sets for which the equality in \eqref{181207e1.3} is achieved are called Salem sets. So far there are a number of constructions of Salem sets, due to Salem \cite{Sal50}, Kaufman \cite{Kau81} (see Bluhm \cite{Blu98} for an exposition), Kahane \cite{Kah85}, Bluhm \cite{Blu96}, \L aba and the third author \cite{LP09}, and so on. Many of these constructions are probabilistic constructions. For instance, Kahane \cite{Kah85} showed that images of compact sets under Brownian motion are almost surely Salem sets.

It is worth mentioning that, in Salem's probabilistic construction of Salem sets \cite{Sal50}, with large probability, the examples there (under certain modifications as in \cite{LP09}) obey assumptions $(A)$ and $(B)$. Moreover, \L aba and the third author \cite{LP09} also provided a probabilistic construction of Salem sets, a large portion of which (under a natural measure) satisfy assumptions $(A)$ and $(B)$. \\

We will discuss a few generalisations of the result of \L aba and the third author \cite{LP09}.  In \cite{CLP16}, Chan, \L aba and the third author generalised \cite{LP09} to higher dimensions. Their result covers  a large class of linear patterns. In particular, they proved: Let $a, b, c$ be three points in the plane that are not co-linear. Let $E\subset \R^2$. Assume that $E$ supports a probability measure $\mu$ satisfying analogues of assumptions $(A)$ and $(B)$ in $\R^2$. Then $E$ must contain three distinct points $x, y, z$ such that the triangle $\Delta xyz$ is similar to $\Delta abc$.

The result of \cite{CLP16} was later generalised to certain nonlinear patterns by Henriot, \L aba and the third author \cite{HLP15}. However, their result does not cover the case of dimension one. For instance, it was pointed out by the authors of \cite{HLP15} that the configuration $(x, x+t, x+t^2)$ with $x, t\in \R$ can not be detected by their method. In the current paper, we provide an affirmative answer to this question. \\

We turn to the proof of Theorem \ref{180817thm1.1}. Let $\tau_0$ be a non-negative smooth bump function supported on the interval $[1, 2]$, and $\tau_l(t):=\tau_0(2^l t)$. It is not difficult to imagine that the trilinear form 
\beq\label{190111e1.5}
\iint \mu(x)\mu(x+t)\mu(x+P(t))\tau_{l}(t)dt dx
\endeq
will play a crucial role in the study of patterns as in our main theorem. However, as $\mu$ is just a measure, the above trilinear form may not be well-defined at the first place. Our first task is to make sense of this trilinear form for every integer $l$ that is large enough. 

Let $s$ be a real number. Define a Sobolev norm 
\beq
\|f\|_{H^s}:=\big(\int_{\R} |\hat{f}(\xi)|^2 (1+|\xi|^2)^{s}d\xi\big)^{1/2}.
\endeq
For $l\in \N$, and two Schwartz functions $f$ and $g$, define 
\beq
T_l(f, g)(x):=\int_{\R} f(x+t)g(x+P(t))\tau_l(t)dt.
\endeq
We will prove 
\begin{prop}\label{180816prop3.1}
There exists a small constant $s_0>0$ and large constant $l_0>0$ and $\gamma_0>0$, depending only on $P(t)$, such that 
\beq\label{180816e3.5}
\|T_l(f, g)\|_{H^{s_0}} \le 2^{\gamma_0 l} \|f\|_{H^{-s_0}}  \|g\|_{H^{-s_0}},
\endeq
for every $l\ge l_0$, and for Schwartz functions $f$ and $g$. 
\end{prop}
This is called a Sobolev improving estimate. We are able to use Proposition \ref{180816prop3.1} to make sense of the  double integral in \eqref{190111e1.5}. Let $\mu$ be a probability measure supported on the interval $[0, 1]$. If we also assume that 
\beq
|\hat{\mu}(\xi)|\le C |\xi|^{-\frac{\beta}{2}}, \text{ for } \beta\in ( 1-s_0,1),
\endeq
and for some constant $C>0$, then $\mu\in H^{-s_0}(\R),$ which is a Sobolev space of some negative order. Recall that Schwartz functions are dense in $H^s$ for every $s\in \R$. By a density argument, we know that the double integral  in \eqref{190111e1.5} is well-defined. To be precise,  we will pick a sequence of Schwartz functions $\{f_n\}_{n=1}^{\infty}$ that convergences to $\mu$ in $H^{-s_0}$, and interpret \eqref{190111e1.5} as 
\beq
\lim_{n\to \infty} \iint f_n(x)f_n(x+t) f_n(x+P(t))\tau_l(t)dtdx.
\endeq
That the above limit exists is guaranteed by Proposition \ref{180816prop3.1}. 
\\

After making sense of the double integral in \eqref{190111e1.5}, we will prove that it is always positive. That is, we will prove 
\begin{thm}\label{main-result}
Under the same assumptions as in Theorem \ref{180817thm1.1}, we are able to find a large integer $l_0\in \N$ and a small positive real number $c_0>0$ such that
\beq\label{180326e1.2}
\iint \mu(x)\mu(x+t)\mu(x+P(t))\tau_{l_0}(t)dt dx\ge c_0.
\endeq
\end{thm}

Intuitively speaking, if $E$ does not contain any three term configuration $(x, x+t, x+P(t))$, then the left hand side of \eqref{180326e1.2} would certainty vanish. However, as we dealing with measures supported on sets of fractional dimensions, we need some extra work to make the above argument rigorous. Roughly speaking, we will construct a Borel measure $\nu$ defined on $[0, 1]^2$ and supported on the set of configurations $(x, x+t, x+P(t))$ with $t>0$, such that $\nu([0, 1]^2)>0$. This will guarantee the existence of the desired polynomial pattern. This will be carried out in the last section. \\

The authors would also like to draw the attention of interested readers to a very recent and interesting development due to Krause \cite{Kra19} on the same problem. \\

{\bf Organisation of paper.} The Sobolev improving estimate in Proposition \ref{180816prop3.1} will be proven in Section \ref{181208section2}. The main tools we will be using include the stationary phase principle and techniques from bilinear oscillatory integrals recently developed by Li \cite{Li13}. In Section \ref{181208section3} we provide a proof of the stationary phase principle that is used in the current paper. Theorem \ref{main-result} will be proven in Section \ref{181208section4} and Section \ref{181208section5}.  The argument that is used in this step relies on the idea of measure decomposition of \L aba and the second author \cite{LP09}, on the Sobolev improving estimate in Proposition \ref{180816prop3.1} and on Bourgain's energy pigeonholing argument from \cite{Bou88}. Finally, in Section \ref{181208section6} we will finish the proof of Theorem \ref{180817thm1.1}. \\

{\bf Notation.} Throughout the paper, we will write $x\lesim y$ to mean that there exists a universal constant $C$ such that $x\le C y$, and $x\simeq y$ to mean that $x\lesim y$ and $y\lesim x$. Moreover, $x\lesim_{M, N} y$ means there exists a constant $C_{M, N}$ depending on the parameters $M$ and $N$ such that $x\le C_{M, N} y$. \\

{\bf Acknowledgements.} This material is based upon work partially supported by the National Science Foundation under Grant No. DMS-1440140 while the authors were in residence at the Mathematical Sciences Research Institute in Berkeley, California, during the Spring semester of 2017. The first author is supported by the NSF under grant Number 1803086. The second author is also partially supported by a direct grant for research (4053295)
from the Chinese University of Hong Kong. The third also is also partially supported by an NSERC discovery grant.

\section{\bf Sobolev improving estimate: Proof of Proposition \ref{180816prop3.1}}\label{181208section2}

In this section we prove Proposition \ref{180816prop3.1}. 
Let $P: \R\to \R$ be a polynomial of degree bigger than one without constant term. We write it as 
\beq
P(t)=a_n t^{\alpha_n}+ \dots + a_2 t^{\alpha_2}+a_1 t^{\alpha_1},
\endeq
with $1\le \alpha_1<\alpha_2<\dots<\alpha_n$. Here we assume that $a_i\neq 0$ for every $i\in \{1, 2, \dots, n\}$. 
Moreover, we assume that $\alpha_1=1$, that is, our polynomial $P$ contains a linear term. The corresponding result for a polynomial without linear term is much easier to prove. This point will be elaborated in a few lines.\\

For each $1\le i\le n$, let $b_i$ be the unique integer such that 
\beq
2^{b_i}\le |a_i| <2^{b_{i}+1}.
\endeq
Let $\Gamma_0$ be a large number that depends on the polynomial $P$.
Let $l_0\in \N$ be the smallest integer such that for every $l\ge l_0$, the following hold:
\beq\label{180816e3.3}
|a_i| 2^{-l\cdot \alpha_i} \le \Gamma^{-1}_0 |a_1 2^{-\alpha_1 l}| \text{ for every  } n\ge i\ge 2,
\endeq
and 
\beq\label{180816e3.4}
|a_i| 2^{-l\cdot \alpha_i} \le \Gamma^{-1}_0 |a_2 2^{-l\cdot \alpha_2}| \text{ for every  } n\ge i\ge 3.
\endeq
In other words, at the scale $t\simeq 2^{-l}$, the monomial $a_1 t$ ``dominates" the polynomial $P(t)$, and $a_2 t^{\alpha_2}$ is the second dominating term. It is not difficult to see that the choice of $l_0$ depends only on $P(t)$. \\

Let us pause and make a remark on the assumption that $\alpha_1=1$. As mentioned above, the case $\alpha_1>1$ is relatively easier to handle. This is because certain curvature (in the sense of oscillatory integrals) appears naturally in this case. To be more precise, under the assumption that $\alpha_1\ge 2$, we first choose $l$ large enough such that  \eqref{180816e3.3} holds, and then the monomial $a_1 t^{\alpha_1}$ dominates the polynomial $P(t)$ at the scale $t\simeq 2^{-l}$. Notice that certain curvature is already present when $a_1 t^{\alpha_1}$ dominates. Hence the requirement \eqref{180816e3.4} becomes redundant. 

However, under the assumption that $\alpha_1=1$, if we only require \eqref{180816e3.3}, then there is no curvature in the dominating term $a_1 t$. This is why we need to further require \eqref{180816e3.4} and find a second dominating term. It is hoped that the curvature in the second dominating term will play an equivalent role. Due to the presence of the linear term $a_1 t$, a number of extra complications will appear. \\

The rest of this section is devoted to the proof of Proposition \ref{180816prop3.1}. Let $h$ be a function in $H^{-s_0}$. We pair it with the left hand side of \eqref{180816e3.5} and study 
\beq
\int_{\R} \Big[\int_{\R} f(x+t)g(x+P(t))\tau_l(t)dt \Big] h(x)dx.
\endeq
Let $\psi_0: \R\to \R$ be a non-negative smooth bump function supported on $[-3, -1]\cup [1, 3]$. Define $\psi_k(\cdot )=\psi_0(\cdot/2^k)$. Moreover, we choose $\psi_0$ such that 
\beq
1=\sum_{k\in \Z} \psi_k(t), \text{ for every } t\neq 0.
\endeq 
For all the three functions $f, g$ and $h$, we apply the non-homogeneous Littlewood-Paley decomposition  $
\mathbbm{1}= \sum_{k\in \N} P_k, $
where $\mathbbm{1}$ denotes the identity operator, and study
\beq
\sum_{k_1, k_2, k_3=0}^{\infty} \iint_{\R^2} P_{k_1}f(x+t)P_{k_2}g(x+P(t))\tau_l(t) P_{k_3}h(x)dt dx.
\endeq
Here 
\beq
P_k f(x):=\int_{\R} e^{ix\xi}\psi_k(\xi)\hat{f}(\xi)d\xi, \text{ if } k>0,
\endeq
and 
\beq
P_0 f(x):=\int_{\R} e^{ix\xi}(\sum_{k\le 0}\psi_k(\xi))\hat{f}(\xi)d\xi.
\endeq
In the following, we work on two cases
\beq
|(k_1-l)-(k_2-l+b_1)|\ge 100 \text{ and } |(k_1-l)-(k_2-l+b_1)|< 100.
\endeq
Let us begin with the first case. Our goal is to prove
\begin{lem}\label{170724lemma3.1}
There exists a constant $\gamma_0\in \N$ depending on $P$, such that under the assumption that $|k_1-k_2-b_1|\ge 100$, we have 
\beq\label{180326e3.13}
\begin{split}
& \left|\iint_{\R^2} P_{k_1}f(x+t)P_{k_2}g(x+P(t))\tau_l(t) P_{k_3}h(x)dt dx \right| \\
\lesim_{N} & 2^{\gamma_0 l} 2^{-N(k_1+k_2+k_3)}\|P_{k_1} f\|_2\|P_{k_2} g\|_2\|P_{k_3}h\|_2,
\end{split}
\endeq
for arbitrarily large $N\in \N$. 
\end{lem} 
Assuming the above lemma, we have 
\beq
\begin{split}
& \sum_{\substack{k_1, k_2, k_3\\|k_1-k_2-b_1|\ge 100}} \left|\iint_{\R^2} P_{k_1}f(x+t)P_{k_2}g(x+P(t))\tau_l(t) P_{k_3}h(x)dt dx\right|\\
&\lesim_l \sum_{k_1, k_2, k_3=0}^{\infty}2^{-10(k_1+k_2+k_3)}\|P_{k_1}f\|_2\|P_{k_2}g\|_2 \|P_{k_3}h\|_2\lesim_l \|f\|_{H^{-s_0}}\|g\|_{H^{-s_0}}\|h\|_{H^{-s_0}},
\end{split}
\endeq
for some $s_0>0$. For instance, we may take $s_0=1$ at this step. 

\begin{proof}[Proof of Lemma \ref{170724lemma3.1}]
The proof is via an integration by parts. Turning to the Fourier side, we can write the left hand side of \eqref{180326e3.13} as 
\beq\label{170724e3.8}
2^{-l} \Big|\iint \widehat{P_{k_1}f}(\xi)\widehat{P_{k_2}g}(\eta) \widehat{P_{k_3}h}(\xi+\eta)\left[\int_{\R} e^{i2^{-l}t\xi+iP(2^{-l}t)\eta}\tau_0(t)dt \right]d\xi d\eta\Big|.
\endeq
First of all, we observe that
\beq
\eqref{170724e3.8}=0 \text{ when } k_3\ge k_1+k_2+10.
\endeq
Hence in the rest of the proof, we assume that $k_3\le k_1+k_2+10$, and it suffices to prove 
\beq
\eqref{170724e3.8} \lesim_l 2^{-N(k_1+k_2)}\|P_{k_1}f\|_2\|P_{k_2}g\|_2 \|P_{k_3}h\|_2.
\endeq
Here $N\in \N$ is a large integer that might vary from line to line. By an integration by parts, we obtain 
\beq
\left|\int_{\R} e^{i2^{-l}t\xi+iP(2^{-l}t)\eta}\tau_0(t)dt \right|\lesim_l 2^{-N\max\{k_1, k_2\}}.
\endeq
Substitute the above pointwise bound into \eqref{170724e3.8}, and apply H\"older's inequality in the $\xi$ and $\eta$ variables. This will finish the proof of the desired estimate. It is easy to track the dependence on $l$ and see that it is polynomial in $2^l$. \\
\end{proof}

From now on we may assume that $|k_1-k_2-b_1|<100$. Without loss of generality, we take $k_1=k_2+b_1$, and consider 
\beq\label{170724e3.12}
\sum_{k_1, k_3=0}^{\infty} \iint_{\R^2} P_{k_1}f(x+t)P_{k_1-b_1}g(x+P(t))\tau_l(t) P_{k_3}h(x)dt dx
\endeq
In the double sum over $k_1$ and $k_3$, we may impose the extra condition that $k_3\le 2 (k_1+|b_1|)$, as otherwise the corresponding term from \eqref{170724e3.12} will simply vanish.
\begin{lem}\label{170724lemma3.2}
There exists a constant $\gamma_0\in \N$ and $\gamma>0$, both of which are allowed to depend on $P$, such that
\beq
\left|\iint_{\R^2} f(x+t)g(x+P(t))\tau_l(t) h(x)dt dx\right| \le  2^{\gamma_0 l}2^{-\gamma k} \|f\|_2 \|g\|_2\|h\|_2,
\endeq
 for every $l\ge l_0$ and every $k\in \N$, under the assumption that  $\supp(\hat{f})\subset \pm [2^k, 2^{k+1}]$ and $\supp(\hat{g})\subset \pm [2^{k-b_1}, 2^{k-b_1+1}]$ and no further assumption on the function $h$.
\end{lem}
Assuming this lemma, we will be able to finish the proof of the desired bilinear estimate. Recall that we need to control  \eqref{170724e3.12}.
By Lemma \ref{170724lemma3.2}, this can be bounded by 
\beq
2^{\gamma_0 l}\sum_{k_1, k_3 \text{ with }k_3\le 2k_1+2|b_1|} 2^{-\gamma k_1}\|P_{k_1}f\|_2 \|P_{k_1-b_1}g\|_2 \|P_{k_3}h\|_2,
\endeq
for some $\gamma>0$, which can be further bounded by 
\beq
\|f\|_{H^{-\gamma/6}}\|g\|_{H^{-\gamma/6}} \|h\|_{H^{-\gamma/6}}.
\endeq
This finishes the proof of the desired estimate. \\

Hence it remains to prove Lemma \ref{170724lemma3.2}. As the constant is allowed to depend on $l$, we can always assume that $k$ is at least some large constant times $\ell$. Turning to the Fourier side, we obtain 
 \beq\label{170724e3.17}
\iint \widehat{P_{k}f}(\xi)\widehat{P_{k-b_1}g}(\eta) \hat{h}(\xi+\eta)\left[\int_{\R} e^{i2^{-l}t\xi+iP(2^{-l}t)\eta}\tau_0(t)dt \right]d\xi d\eta.
\endeq
Write 
\beq
P(t)=a_1 t+Q(t)=a_1 t+a_2 t^{\alpha_2}+ R(t).
\endeq
The derivative of the phase function in \eqref{170724e3.17} is given by 
\beq
2^{-l} \xi+ a_1 2^{-l} \eta+ a_2 \alpha_2 2^{-\alpha_2 l} t^{\alpha_2-1} \eta + 2^{-l} R'(2^{-l}t)\eta.
\endeq
From the first order derivative of the phase function, we are still not able to locate the critical point. To do so, we apply a more refined frequency decomposition to $f$ and $g$. For a fixed integer $\Delta$, let $\psi_{k, l, \Delta}: \R\to \R$ be a non-negative smooth function supported on $[2^k+\Delta\cdot 2^{k-\gamma_0 l}, 2^k+ (\Delta+2) 2^{k-\gamma_0 l}]$ such that 
\beq\label{180320e3.29}
\psi_k(\xi)=\sum_{\Delta\in \Z} \psi_k(\xi)\psi_{k, l, \Delta}(\xi), \text{ for every } \xi\in \R.
\endeq 
That is, $\{\psi_{k, l, \Delta}\}_{\Delta\in \Z}$ forms a partition of unity on the support of $\psi_k$. Moreover, the sum in \eqref{180320e3.29} is indeed a finite sum, and the number of non-zero terms is about $2^{\gamma_0 l}$. Here $\gamma_0$ is some large number that is to be chosen. For convenience, we will allow $\gamma_0$ to change from line to line, unless otherwise stated. \\

We write \eqref{170724e3.17} as 
\beq
\sum_{\Delta_1, \Delta_2\in \Z} \iint \big[\widehat{P_{k}f}(\xi)\psi_{k, l, \Delta_1}(\xi)\big] \big[ \widehat{P_{k-b_1}g}(\eta)\psi_{k-b_1, l, \Delta_2}(\eta)\big] \hat{h}(\xi+\eta)\left[\int_{\R} e^{i2^{-l}t\xi+iP(2^{-l}t)\eta}\tau_0(t)dt \right]d\xi d\eta
\endeq
Notice that in the above sum, we have about $2^{2\gamma_0 l}$ terms that may be non-zero. As the implicit constant is allowed to depend on $l$, it suffices to bound each term separately. Moreover, by the stationary phase principle, we only need to care about those terms whose phase functions admit critical points. In other words, we only need to care about those $\Delta_1$ and $\Delta_2$ such that 
\beq\label{180322e3.31}
|\frac{2^{-l}\xi+a_1 2^{-l}\eta}{a_2 \alpha_2 2^{-\alpha_2 l}\eta}|\simeq 1,
\endeq
for some $\xi\in \text{supp}(\psi_{k, l, \Delta_1})$ and $\eta\in \text{supp}(\psi_{k-b_1, l, \Delta_2})$. Fix such $\Delta_1$ and $\Delta_2$, by the mean value theorem, we actually know that \eqref{180322e3.31} holds true for every $\xi\in \text{supp}(\psi_{k, l, \Delta_1})$ and $\eta\in \text{supp}(\psi_{k-b_1, l, \Delta_2})$, if we choose $\gamma_0$ large enough, depending on $P(t)$. \\

After this reduction, what we need to prove becomes
\beq\label{180322e3.32}
\Big| \iint \hat{f}(\xi) \hat{g}(\eta) \hat{h}(\xi+\eta)\left[\int_{\R} e^{i2^{-l}t\xi+iP(2^{-l}t)\eta}\tau_0(t)dt \right]d\xi d\eta\Big| \le 2^{\gamma_0 l}2^{-\gamma k} \|f\|_2 \|g\|_2 \|h\|_2,
\endeq
under the assumption that 
\beq
\hat{f}=\widehat{P_k f}\cdot  \psi_{k, l, \Delta_1}, \,\,\, \hat{g}=\widehat{P_{k-b_1}g}\cdot \psi_{k-b_1, l, \Delta_2}
\endeq
and that \eqref{180322e3.31} holds for every $\xi\in \text{supp}(\hat{f})$ and $\eta\in \text{supp}(\hat{g})$.
Let $t_c\in [1, 2]$ be the critical point of the phase function; that is, 
\beq
\xi+P'(2^{-l}t_c)\eta=0.
\endeq
We will prove the following approximation formula. 
\begin{lem}[Method of stationary phase]\label{180719lem3.3}
Under the above notation, we have 
\beq\label{170724e3.18}
\int_{\R} e^{i2^{-l}t\xi+iP(2^{-l}t)\eta}\tau_0(t)dt=  a(\xi, \eta) \eta^{-1/2} e^{i \Psi(\xi, \eta)}+O_l(\frac{1}{|\eta|}),
\endeq
with 
\beq
a(\xi, \eta):=(2^{-2l} P''(2^{-l}t_c) )^{-1/2}\tau_0(t_c) 
\endeq
and
\beq
 \Psi(\xi, \eta):=2^{-l}t_c \xi+P(2^{-l} t_c)\eta.
 \endeq
 Moreover, 
 \beq
 O_l(\frac{1}{|\eta|}) \le 2^{\gamma_0 l} \frac{1}{|\eta|}.
 \endeq
 \end{lem}
Lemma \ref{180719lem3.3} will be proved in Section \ref{181208section3}. Substituting \eqref{170724e3.18} into \eqref{180322e3.32} gives rise to two terms. Let us first estimate the contribution from the term containing $O_l(\frac{1}{|\eta|})$. We bound it by 
 \beq\label{170724e3.19}
 \begin{split}
& \iint \left|\hat{f}(\xi)\hat{g}(\eta) \hat{h}(\xi+\eta)\frac{1}{|\eta|}\right|d\xi d\eta \lesim_l 2^{-k} \iint \left|\hat{f}(\xi)\hat{g}(\eta) \hat{h}(\xi+\eta)\right|d\xi d\eta.\end{split}
\endeq
By H\"older's inequality, the last term can be bounded by 
\beq
2^{-k/2} \|f\|_2 \|g\|_2\|h\|_2.
\endeq
So far we have managed to control the contribution from the second term on the right hand side of \eqref{170724e3.18}. \\

Now we turn to the first term on the right hand side of \eqref{170724e3.18}. The corresponding term we need to handle is 
\beq
\begin{split}
& \iint \hat{f}(\xi)\hat{g}(\eta) \hat{h}(\xi+\eta) \frac{a(\xi, \eta)}{\sqrt{\eta}} e^{i \Psi(\xi, \eta)} d\xi d\eta.
\end{split}
\endeq
We apply a change of variables $\xi\to 2^k\xi, \eta\to 2^{k}\eta$. We also rename $f, g, h$ for simplicity. It suffices to prove 
\beq\label{181106e3.40}
\Big|\iint f(\xi)g(\eta) h(\xi+\eta) a(\xi, \eta) e^{i 2^k \Psi(\xi, \eta)} d\xi d\eta\Big| \lesim 2^{-\gamma k} \|f\|_2\|g\|_2\|h\|_2,
\endeq
for every function $a: \R^2\to \R$ with $\|a\|_{C^4}\lesim 1$,  and for functions $f$ supported on $\pm [1+\Delta_1 2^{-\gamma_0 l}, 1+(\Delta_1+2) 2^{-\gamma_0 l}]$ and $g$ supported on $\pm [2^{-b_1}+\Delta_2 2^{-b_1-\gamma_0 l}, 2^{-b_1}+(\Delta_2+2) 2^{-b_1-\gamma_0 l}]$. Here $\Delta_1$ and $\Delta_2$ are two positive integers that are smaller than $2^{\gamma_0 l}$. Moreover, they are chosen such that \eqref{180322e3.31} holds for every $\xi\in \supp(f)$ and $\eta\in \supp(g)$.
\begin{claim}\label{181106claim3.5}
There exist integers $C_P$, $C_P'$ depending only on $P$ and intervals $J_1, \dots, J_{C_P}\subset \R$ of length $2^{-\gamma k/C_P'}$, such that whenever $\xi/\eta\not\in J_{\iota}$ for any $\iota$, we have 
\beq
|\partial_{\xi}\partial_{\eta}(\partial_{\xi}-\partial_{\eta}) \Psi|\gtrsim 2^{-\gamma k}.
\endeq
The implicit constant is allowed to depend on $P$, and can be taken to be polynomial in $2^{l}$.
\end{claim}
The proof of the claim is postponed to the end of this section. Let $\widetilde{a}: \R\to \R$ be a smooth bump function taking value one on each $2J_{\iota}$ such that $\|\widetilde{a}\|_{C^4}\lesim 2^{4\gamma k}$. To prove \eqref{181106e3.40}, we will decompose $a(\xi, \eta)=a(\xi, \eta) \widetilde{a}(\xi/\eta) +a(\xi, \eta)(1-\widetilde{a}(\xi/\eta)) $ and control the two resulting terms separately. 
For the former term, the oscillation from $e^{i2^k \Psi}$ no longer plays any role, and we simply put the absolute value sign inside the integral and obtain 
\beq
\iint \big| f(\xi)g(\eta) h(\xi+\eta) a(\xi, \eta)\widetilde{a}(\xi/\eta) \big| d\xi d\eta.
\endeq
By Cauchy-Schwarz, this can be easily bounded by $2^{-\gamma k/C_P'}\|f\|_2 \|g\|_2\|h\|_2.$ To control the latter term, it suffices to prove 
\begin{lem}\label{181106elmma3.6}
For every small positive $\gamma>0$, every function $a: \R^2\to\R $ supported on $[-1, 1]^2$ with $\|a\|_{C^4}\le 2^{\gamma k}$, every $\Psi: \R^2\to \R$ with 
\beq
|\partial_{\xi}\partial_{\eta}(\partial_{\xi}-\partial_{\eta}) \Psi|\gtrsim 2^{-\gamma k} \text{ and } \|\Psi\|_{C^4}\lesim 1,
\endeq
we have 
\beq\label{181106e3.44}
\Big|\iint f(\xi)g(\eta) h(\xi+\eta) a(\xi, \eta) e^{i 2^k \Psi(\xi, \eta)} d\xi d\eta\Big| \lesim 2^{-\gamma k} \|f\|_2\|g\|_2\|h\|_2.
\endeq
Here taking $\gamma=10^{-5}$ is more than enough. 
\end{lem}
\begin{proof}[Proof of Lemma \ref{181106elmma3.6}.]
This lemma is essentially due to Li \cite{Li13}. Here we need to keep track of the dependence on norms of $a$, on its support, and so on. Oscillatory integrals of the form \eqref{181106e3.44} have also been extensively studied in Xiao \cite{Xiao17} and Gressman and Xiao \cite{GX16}. 

We start the proof. By applying the triangle inequality, it suffices to prove \eqref{181106e3.44} with a better gain $2^{-3\gamma k}$ in place of $2^{-\gamma k}$, for every function $g$ supported on an interval of length $2^{-2\gamma k}$. By a change of variable and by applying Cauchy-Schwarz, it is enough to prove 
\beq
\Big\|\int f(\xi-\eta)g(\eta) a(\xi-\eta, \eta) e^{i 2^k \Psi(\xi-\eta, \eta)}d\eta\Big\|^2_{L^2_{\xi}} \lesim 2^{-6\gamma k} \|f\|^2_2\|g\|^2_2.
\endeq
We expand the square on the left hand side. After a change of variable, we obtain 
\beq\label{181106e3.46}
\int_{|\tau|\le 2^{-\gamma k}} \iint_{\R^2} e^{i2^k[\Psi(\xi, \eta)-\Psi(\xi-\tau, \eta+\tau)]} F_{\tau}(\xi) G_{\tau}(\eta) a'_{\tau}(\xi, \eta)d\xi d\eta d\tau,
\endeq
for some new compactly supported amplitude $a'_{\tau}$. Moreover, $F_{\tau}(\cdot):=f(\cdot)\bar{f}(\cdot-\tau)$ and $G_{\tau}(\cdot):=g(\cdot)\bar{g}(\cdot+\tau)$. By the mean value theorem, it is easy to see that 
\beq
\Big|\partial_{\xi}\partial_{\eta} \Big(\Psi(\xi, \eta)-\Psi(\xi-\tau, \eta+\tau) \Big)\Big|\gtrsim 2^{-\gamma k}|\tau|.
\endeq
To proceed, we need 
\begin{lem}\label{181106elmma3.7}
For every small positive $\gamma>0$, every function $a': \R^2\to \R$ supported on $[-1, 1]^2$ with $\|a'\|_{C^4}\le 2^{\gamma k}$, every $\Xi: \R^2\to \R$ with 
\beq
\big|\partial_{\xi}\partial_{\eta}\Xi \big|\gtrsim 2^{-7\gamma k} \text{ and } \|\Xi\|_{C^4}\lesim 1,
\endeq
we have 
\beq
\Big|\iint F(\xi)G(\eta) a'(\xi, \eta) e^{i2^k \Xi(\xi, \eta)}d\xi d\eta \Big|\lesim 2^{-6\gamma k} \|F\|_2 \|G\|_2.
\endeq
Again taking $\gamma=10^{-5}$ is more than enough. 
\end{lem}
To control \eqref{181106e3.46}, we split the integral in $\tau$ into two parts: 
\beq
\int_{|\tau|\le 2^{-6\gamma k}}+\int_{|\tau|\ge 2^{-6\gamma k}}.
\endeq
Regarding the former term, we apply the triangle inequality and Cauchy-Schwarz to bound it by $2^{-6\gamma k}\|f\|_2^2\|g\|_2^2$. Regarding the latter term, we apply Lemma \ref{181106elmma3.7} and bound it by 
\beq
2^{-6\gamma k}\int_{|\tau|\ge 2^{-6\gamma k}} \|F_{\tau}\|_2\|G_{\tau}\|_2d\tau
\endeq
By applying Cauchy-Schwarz, this is bounded by $2^{-6\gamma k}\|f\|_2^2\|g\|_2^2.$ This finishes the proof of Lemma \ref{181106elmma3.6}.
\end{proof}

\begin{proof}[Proof of Lemma \ref{181106elmma3.7}]
This lemma is essentially due to H\"ormander \cite{Hor73}. By Cauchy-Schwarz, it suffices to prove 
\beq\label{181106e3.52}
\Big\|\int F(\xi) a'(\xi, \eta)e^{i2^k \Xi(\xi, \eta)}d\xi\Big\|^2_2 \lesim 2^{-12\gamma k}\|F\|_2^2.
\endeq
By the triangle inequality, it suffices to prove \eqref{181106e3.52} with a better gain $2^{-30\gamma k}$ in place of $2^{-12 \gamma k}$, for every function $F$ supported on an interval of length $2^{-8\gamma k}$. We expand the square on the left hand side and obtain 
\beq
\iint \sigma_k(\xi_1, \xi_2) F(\xi_1)\bar{F}(\xi_2)d\xi_1 d\xi_2,
\endeq
where 
\beq
\sigma_k(\xi_1, \xi_2):=\int e^{i2^k (\Xi(\xi_1, \eta)-\Xi(\xi_2, \eta))}a'(\xi_1, \eta) \bar{a'}(\xi_2, \eta)d\eta.
\endeq
By the mean value theorem, we observe that 
\beq
\big|\partial_{\eta}(\Xi(\xi_1, \eta)-\Xi(\xi_2, \eta))\big|\gtrsim 2^{-7\gamma k}|\xi_1-\xi_2|.
\endeq
By applying integration by parts twice, we obtain 
\beq
|\sigma_k(\xi_1, \xi_2)|\lesim \min\{2^{2\gamma k}, 2^{-2k+50\gamma k}|\xi_1-\xi_2|^{-6}\}.
\endeq
By Schur's test, this gives us the desired bound if we choose $\gamma$ small enough. This finishes the proof of Lemma \ref{181106elmma3.7}.
\end{proof}

\begin{proof}[Proof of Claim \ref{181106claim3.5}.]
Recall that $t_c(\xi, \eta)$ is defined via
\beq
\xi+a_1 \eta+\eta Q'(2^{-l}t_c)=0.
\endeq
Moreover, 
\beq
\begin{split}
\Psi(\xi, \eta)& =2^{-l} (\xi+a_1 \eta) t_c+ \eta Q(2^{-l}t_c)\\
		& = (\xi+a_1 \eta) (Q')^{-1}\Big(-\frac{\xi+a_1 \eta}{\eta}\Big)+\eta Q\Big((Q')^{-1}\Big(-\frac{\xi+a_1 \eta}{\eta}\Big) \Big).
\end{split}
\endeq
Here $(Q')^{-1}$ means the inverse of the derivative of $Q$. By a direct calculation, we obtain 
\beq
\begin{split}
& \Big|\partial_{\xi}\partial_{\eta}(\partial_{\xi}-\partial_{\eta})\Psi(\xi, \eta) \Big|\\
& \approx \Big|(2\rho+1)\Big(Q''\Big((Q')^{-1}(-\rho-a_1)\Big)\Big)^2 +(\rho^2+\rho)Q'''\Big((Q')^{-1}(-\rho-a_1) \Big)\Big|,
\end{split}
\endeq 
where $\rho:=\xi/\eta$. By changing $\rho$ to $-\rho-a_1$, it is equivalent to consider 
\beq\label{181202e3.60}
\Big|-(2\rho+2 a_1-1)\Big(Q''\Big((Q')^{-1}(\rho)\Big)\Big)^2 +(\rho+a_1)(\rho+a_1-1)Q'''\Big((Q')^{-1}(\rho) \Big)\Big|
\endeq
Recall that $Q(t)=a_2 t^{\alpha_2}+R(t)$, where $a_2\neq 0$ and $R(t)$ can be viewed as a remainder term compared with $a_2 t^{\alpha_2}$ when $t\approx 2^{-l}$. Denote $s:=(Q')^{-1}(\rho)$. Then \eqref{181202e3.60} becomes 
\beq
\Big|-(2Q'(s)+2 a_1-1)\Big(Q''(s)\Big)^2 +(Q'(s)+a_1)(Q'(s)+a_1-1)Q'''(s)\Big|.
\endeq
The highest order term in the last display is given by
\beq
a_n^3 \alpha_n^3 s^{3\alpha_n-5} \Big(-2(\alpha_n-1)^2+(\alpha_n-1)(\alpha_n-2) \Big).
\endeq 
Notice that the coefficient does not vanish. From this, Claim \ref{181106claim3.5} follows immediately by choosing $\gamma$ small enough. 
\end{proof}

\section{\bf Stationary phase principle: Proof of Lemma \ref{180719lem3.3}}\label{181208section3}

Our goal in this section is to prove an asymptotic formula for 
\beq
\int_{\R} e^{i2^{-l}t\xi+iP(2^{-l}t)\eta}\tau_0(t)dt=2^l \int_{\R} e^{i t\xi+iP(t)\eta}\tau_0(2^l t)dt
\endeq
We follow the proof of Proposition 3 on Page 334 of Stein \cite{Ste93}.  Define 
\beq
\Phi_{\xi, \eta}(t):=t\xi+P(t)\eta.
\endeq
Recall some notation 
\beq
P(t)=a_1 t+Q(t)=a_1 t+a_2 t^{\alpha_2}+ R(t).
\endeq
Let $t_c\in (2^{-l-1}, 2^{-l+1})$ be the critical point of the phase function, that is, 
\beq
\xi+P'(t_c)\eta=\xi+(a_1+Q'(t_c))\eta=0.
\endeq
We expand the phase function about $t_c$:
\beq
\Phi_{\xi, \eta}(t)=\Phi_{\xi, \eta}(t_c)+ \frac{1}{2}Q''(t_c)\eta (t-t_c)^2 +O_l (|t-t_c|^3)\cdot \eta.
\endeq
Here 
\beq
O_l (|t-t_c|^3) \le C_P 2^{l} |Q''(t_c)||t-t_c|^3,
\endeq
where $C_P$ is a large constant depending only on $P$. 
Let $\vartheta$ be  a non-negative even smooth  function supported on $[-2, 2]$, constant on $[-1,1]$, and monotone on $[1,2]$. We normalize it such that  $\widehat{\vartheta}(0)=1$ and denote $\vartheta_\ell(x):=2^{\ell}\vartheta(2^{\ell}x)$. We write 
\beq
\begin{split}
\int_{\R} e^{i t\xi+iP(t)\eta}\tau_0(2^l t)dt&=\int_{\R} e^{i t\xi+iP(t)\eta}\tau_0(2^l t)\vartheta(2^{l+10 C_P} (t-t_c))dt\\
&+ \int_{\R} e^{i t\xi+iP(t)\eta}\tau_0(2^l t)\Big(1-\vartheta(2^{l+10 C_P} (t-t_c))\Big)dt=: I+II.
\end{split}
\endeq
The phase function in term $II$ does not admit any critical point. Hence by integration by parts, we obtain 
\beq
|II| \le 2^{\gamma_0 l} \frac{1}{|\eta|}. 
\endeq
For term $I$, we write it as 
\beq\label{180817e6.9}
e^{i\Phi_{\xi, \eta}(t_c)}\int_{\R}e^{i\Phi_{\xi, \eta}(t)-i\Phi_{\xi, \eta}(t_c)} \widetilde{\vartheta_l}(t)dt,
\endeq
where 
\beq
\widetilde{\vartheta_l}(t):=\tau_0(2^l t)\vartheta(2^{l+10 C_P} (t-t_c)).
\endeq
The support of the function $\widetilde{\vartheta_l}$ is chosen to be so small such that the change of variable 
\beq
(t-t_c)^2 +\frac{2}{Q''(t_c)}\cdot O_l (|t-t_c|^3)\to  y^2
\endeq
becomes valid. Under this change of variable, \eqref{180817e6.9} turns to 
\beq
e^{i\Phi_{\xi, \eta}(t_c)} \int_{\R} e^{i\frac{1}{2}Q''(t_c) \eta y^2} \vartheta'_l(y)dy,
\endeq
for some new smooth truncation function $\vartheta'_l$. We split the last expression into three terms: 
\beq
\begin{split}
& e^{i\Phi_{\xi, \eta}(t_c)} \int_{\R} e^{i\frac{1}{2}Q''(t_c) \eta y^2} e^{-y^2} (e^{y^2}\vartheta'_l(y)-\vartheta'_l(0)) \vartheta''_l(y)dy\\
&+e^{i\Phi_{\xi, \eta}(t_c)} \int_{\R} e^{i\frac{1}{2}Q''(t_c) \eta y^2} e^{-y^2} \vartheta'_l(0) (\vartheta''_l(y)-1)dy\\
&+e^{i\Phi_{\xi, \eta}(t_c)} \int_{\R} e^{i\frac{1}{2}Q''(t_c) \eta y^2} e^{-y^2} \vartheta'_l(0)dy,
\end{split}
\endeq
where $\vartheta''_l$ is a compactly supported  smooth function and is $1$ on the support of $\vartheta'_l$. These three terms will be called $I_1, I_2$ and $I_3$ and will be handled separately. \\

By the triangle inequality and an integration by parts argument, it is not difficult to see that 
\beq
|I_1|+ |I_2|\le 2^{\gamma_0 l} |\eta|^{-1}. 
\endeq
In the end, one just needs to observe that 
\beq
\int_{\R} e^{i\lambda t^2} e^{-t^2}dt=e_0 \lambda^{-1/2}+ O(\lambda^{-3/2})
\endeq
for some universal constant $e_0$. See equation (9) on Page 335 of Stein \cite{Ste93}. This finishes the proof of Lemma \ref{180719lem3.3}. \\

\section{\bf Positivity of the double integral: Proof of Theorem \ref{main-result}}\label{181208section4}

In this section, we prove Theorem \ref{main-result}. We follow the idea of \L aba and the third author \cite{LP09} and decompose 
\beq
\mu=\mu_1+\mu_2,
\endeq
with 
\beq
\mu_1(x)\le A \cdot 2^{6B} C_1,
\endeq
where $A$ is a large absolute constant. Here $\mu_1$ is obtained by convolving $\mu$ with a Fejer kernel. See page 442 of Laba and the third author \cite{LP09}. Here we make a remark that this is only place where one applies the assumption $(A)$ in \eqref{180817e1.1}. Also, in their decomposition, it is possible to choose $\mu_1$ so that 
\beq
\mu_1\ge 0 \text{ and } \int_0^1 \mu_1(x)dx=1.
\endeq
Moreover, we have $\widehat{\mu_2}(0)=0$, and 
\beq
\widehat{\mu_2}(n)=\min\left(1, \frac{|n|}{2N+1} \right) \hat{\mu}(n),
\endeq
where 
\beq\label{181208e4.5}
N=C_2^{-1} e^{\frac{1}{1-\alpha}}.
\endeq 
\begin{lem}\label{bourgain-lemma}
There exists $l_0\in \N$ and $c_0>0$ depending only on $C_1, C_2, B$, $\beta$ and the polynomial $P$ such that 
\beq
\iint \mu_1(x)\mu_1(x+t)\mu_1(x+P(t))\tau_{l_0}(t)dt dx\ge c_0.
\endeq
\end{lem}
The proof of Lemma \ref{bourgain-lemma} is based on Bourgain's energy pigeonholing argument \cite{Bou88} and the Sobolev improving estimate in Proposition \ref{180816prop3.1}. We postpone its proof to the next section. 

After finding $l_0$ and $c_0$, we will pick $\alpha$ to be sufficiently close to one, and prove that 
\beq\label{170710e3.10}
\left|\iint \mu_{i_1}(x)\mu_{i_2}(x+t)\mu_{i_3}(x+P(t))\tau_{l_0}(t)dt dx\right|\le c_0/8,
\endeq
when $(i_1, i_2, i_3)\neq (1, 1, 1)$. For the sake of simplicity, let us assume that we are working with $(i_1, i_2, i_3)=(1, 1, 2)$.
The proofs of the other cases are similar. In previous sections, we proved that 
\beq\label{170710e3.11}
\left|\iint \mu_{1}(x)\mu_{1}(x+t)\mu_{2}(x+P(t))\tau_{l_0}(t)dt dx\right|\le C_{l_0} \|\mu_1\|^2_{H^{-s_0}} \|\mu_2\|_{H^{-s_0}},
\endeq
for some $s_0>0$ depending only on the polynomial $P$. By the definition of $\mu_1$ and the assumption on $\mu$, we have 
\beq
\|\mu_1\|^2_{H^{-s_0}}\le \|\mu\|^2_{H^{-s_0}} \le C_2^2 (1-\alpha)^{-2B} \sum_{k\ge 1} |k|^{-\beta} |k|^{-2s_0}  
\endeq
Next we turn to the term $\|\mu_2\|_{H^{-s_0}}$.
\beq
\begin{split}
\|\mu_2\|^2_{H^{-s_0}}&\le C_2^2 (1-\alpha)^{-2B} \left(\sum_{1\le k\le 2N} \frac{k^2}{(2N+1)^2} k^{-\beta-2s_0} +\sum_{k>2N} k^{-\beta-2s_0} \right)\\
		&\lesim C_2^2 (1-\alpha)^{-2B} \left( \frac{N^{3-\beta-2s_0}}{N^2} +N^{1-\beta-2s_0}\right) \le C_2^2 (1-\alpha)^{-2B} N^{1-\beta-2s_0}
\end{split}
\endeq
Combined with \eqref{170710e3.11}, we obtain 
\beq
\left|\iint \mu_{1}(x)\mu_{1}(x+t)\mu_{2}(x+P(t))\tau_{l_0}(t)dt dx\right|\le C_{l_0} C_2^{10} (1-\alpha)^{-10B} N^{(1-\beta-2s_0)/2}
\endeq
Recall \eqref{181208e4.5}. 
If we choose $\alpha$ close enough to one, depending on all the other parameters, we will be able to conclude \eqref{170710e3.10}.

\section{\bf Proof of Lemma \ref{bourgain-lemma}}\label{180817section5}\label{181208section5}

Before we start the proof of Lemma \ref{bourgain-lemma}, we state a preliminary lemma. 
Recall the definition of $\vartheta$ in Section \ref{181208section3}.

\begin{lem}[Bourgain \cite{Bou88}]\label{lemma:bourgain}
For a non-negative function $f$ supported on $[0,1]$ and  $k,l\in \N$ we have 
\begin{equation*}
\int_0^1  f (f* \vartheta_{k}) (f* \vartheta_{\ell}) \ge c_0 \Big(\int_0^1 f\Big)^3
\end{equation*}
for some constant $c_0>0$ depending only on the choice of $\vartheta$.
\end{lem}
The proof of this lemma was omitted in \cite{Bou88}. For a proof, we refer to \cite{DGR17}.\\

 In this section, we will use $f$ to stand for $\mu_1$. Hence $f$ is a function satisfying 
\beq
\int f=1 \text{ and } 0\le f\le A\cdot 2^{6B} C_1=: M.
\endeq
For simplicity, we assume $||\tau_0||_1 = 1$ and change the notation a bit by taking 
\beq
\tau_l(t)=2^l \tau_0(2^l t) \text{ instead of } \tau_l(t)=\tau_0(2^l t).
\endeq
We also need to show that $l_0$ can be bounded from above by a number which depends only on $C_1$, $C_2$, $B$, $\beta$ and $P$. Denote 
\beq
\Lambda_{l}=\iint f(x)f(x+t)f(x+P(t))\tau_{l}(t)dt dx.
\endeq
For $\ell', \ell, \ell'' \in \N$ with $1\leq \ell'\leq \ell \leq \ell''$, we have 
\begin{align*}
 \Lambda_{l}& = \int_0^1\int_0^1  f(x)f(x+t)f(x+P(t))\tau_{\ell}(t)dxdt\\ 
 &= I_1 + I_2 + I_3,
\end{align*}
 where
\begin{align*}
I_1 & =\int_0^1\int_0^1  f(x) f(x+t)(f*\vartheta_{\ell'})(x+P(t))\tau_{\ell}(t)dxdt,\\ 
I_2 & = \int_0^1\int_0^1  f(x) f(x+t)(f*\vartheta_{\ell''}-f*\vartheta_{\ell'})(x+P(t))\tau_{\ell}(t)dxdt,\\
I_3 & = \int_0^1\int_0^1  f(x)f(x+{t})(f-f*\vartheta_{\ell''})(x+P(t)) \tau_{\ell}(t)dxdt. 
\end{align*}

We analyze each of the  terms separately.
Splitting $f-f*\vartheta_{\ell''}$ into Littlewood-Paley pieces and applying Lemma \ref{170724lemma3.2} and Lemma \ref{170724lemma3.1}, it follows that for some $\sigma>0$  we have
\begin{align*}
 |I_3| \lesim_M   2^{\gamma_0\ell-\sigma \ell''} \|f\|_{L^2(\R)}^2 \leq 2^{-100}c_0,
\end{align*}  
where the last inequality holds provided that $\ell''$ is taken large enough with respect to $\ell$. Here $c_0$ is the constant from Lemma \ref{lemma:bourgain}.\\

To estimate $I_2$ we apply the Cauchy-Schwarz inequality in $x$, which yields
\begin{align*}
|I_2| &\le \int_0^1 \|f(x)f(x+{t})\|_{L_x^2} \|(f*\vartheta_{\ell''}-f*\vartheta_{\ell'})(x+P(t))\|_{L_x^2} \tau_{\ell}(t) dt\\
&\lesim_M \|f*\vartheta_{\ell''}-f*\vartheta_{\ell'}\|_2. 
\end{align*}
Passing to the last line we bounded the $L^\infty$ norm of $f$ by $M$ and the $L^1$ norm of $\tau_{\ell}$ by one.\\

To estimate $I_1$, we compare it to
\begin{align*}
I_4&=\int_0^1\int_0^1  f(x)f(x+{t})(f*\vartheta_{\ell'})(x)    \tau_{\ell}(t)dx dt\\
&=\int_0^1 f(x) (f*\vartheta_{\ell'})  (x) (f* \tau_{\ell})(x)dx.
\end{align*}
Consider the difference
\begin{equation*}
I_4-I_1= \int_0^1\int_0^1  f(x) f(x+{t})\big( (f*\vartheta_{\ell'})  (x)-(f*\vartheta_{\ell'}) (x+P(t))\big ) \tau_{\ell}(t)dx dt.
\end{equation*}
By the mean value theorem we obtain  
\begin{equation*}
|(f*\vartheta_{\ell'})(x)-(f*\vartheta_{\ell'})(x+P(t))|\leq  2^{\ell'}\|f*(\vartheta_{\ell'})'  \|_\infty|P(t)| \leq M  {2^{\ell'-\ell+1}},
\end{equation*}
whenever $t$ is in the support of $\tau_{\ell}$.
Choosing $\ell$ large enough with respect to $\ell'$ gives
\begin{equation*}
|I_4-I_1|\leq  2^{-100}c_0.
\end{equation*}
We return to analyzing the term $I_4$, which we write as
\begin{align}\label{termI41}
I_4 &= \Big( \int_0^1  f(x) (f*\vartheta_{\ell'})(x) \big ((f* \tau_{\ell})(x)-(f* \vartheta_{\ell'})(x) \big )dx\Big)\\ \label{termI42}
&+ \Big( \int_0^1  f(x) (f*\vartheta_{\ell'})  (x) (f* \vartheta_{\ell'})(x)dx \Big ) 
\end{align}
By Lemma \ref{lemma:bourgain}, the term \eqref{termI42} 
is bounded from below by $c_0$. 
For \eqref{termI41} we use   the triangle inequality and  Young's convolution inequality to estimate
\begin{align} \nonumber
& \|f*\tau_{\ell }-f*\vartheta_{\ell' }\|_2\\ \label{termI421}
& \lesim_M  \|(f*\tau_{\ell }*\vartheta_{\ell'' })-(f*\vartheta_{\ell' }*\tau_{\ell })\|_2\\ \label{termI422}
&   +    \|\tau_{\ell }-(\tau_{\ell }*\vartheta_{\ell'' })\|_1\\
\label{termI423}
& + \|\vartheta_{\ell' }-(\vartheta_{\ell' }*\tau_{\ell })\|_1
  \end{align}
By another application of Young's convolution inequality in \eqref{termI421}, we bound the last display by 
\begin{align*}
\lesim_M \|(f*\vartheta_{\ell'' })-(f*\vartheta_{\ell' })\|_2
  +   \|\tau_{\ell}-(\tau_{\ell}*\vartheta_{\ell''})\|_1 + \|\vartheta_{\ell'}-(\vartheta_{\ell'}*\tau_{\ell})\|_1.
\end{align*}
By the mean value theorem, the second and third term are bounded from above by    $2^{-100}c_0 $  provided $\ell''$ is chosen large enough with respect to $\ell$, and $\ell$ large enough with respect to $\ell'$. This in turn bounds \eqref{termI41}  from above by 
$$ \|f*\vartheta_{\ell'' }-f*\vartheta_{\ell' }\|_2 + 2^{-99}c_0.$$
 
From   the estimates for the terms $I_1,I_2,I_3,I_4$ and $I_4-I_1$     we  obtain
\begin{align*}
c_0 \leq \Lambda_l + C_M \|f*\vartheta_{\ell'} -f*\vartheta_{\ell''}\|_2+ 2^{-90}c_0.
\end{align*}
Here $C_M$ is a large constant that depends only on $M$. For instance it suffices to take $C_M=M^{10}$. Therefore, we either have  
 $\Lambda_{l}> 2^{-10} c_0$,
or 
\begin{equation*}
\|f*\vartheta_{\ell'} -f*\vartheta_{\ell''}\|_2> 2^{-10}C_{M}^{-1}c_0.
\end{equation*}
By the preceding discussion we can construct  a sequence $\{100=\ell_0<\ell_1<\cdots < \ell_k <\cdots\}\subseteq\N$, which is independent of $f$ and satisfies $\ell_{k+1} \le C \ell_k$ for some sufficiently large constant $C$ such that for each $k$
either $$\Lambda_{\ell_k}> 2^{-10} c_0 $$ or 
\begin{align}\label{secondalt}
\|f*\vartheta_{\ell_k} -f*\vartheta_{{\ell}_{k+1}}\|_2> 2^{-10} c_0 C_M^{-1}.
\end{align}
 Observe that  for any $K\geq 0$ one has
\begin{equation}\label{170725e6.10}
\sum_{k=0}^K\Big( \|f*\vartheta_{\ell_k} -f*\vartheta_{\ell_{k+1}}\|^2_2\Big ) \leq C_0 \|f\|^2_2\leq C_0 M^2
\end{equation}
with $C_0$ independent of $K$ and $f$. Let us fix $K> C_02^{100}c_0^{-2}C_M^2 M^4$. 
If \eqref{secondalt} holds for all $0<k \le K$, then \eqref{170725e6.10} yields $K\leq C_02^{100}c_0^{-2} C_M^2 M^4$, which is a contradiction. Thus, for some $0\le k \le K$ we necessarily have $\Lambda_{\ell_k}>2^{-10}  c_0$. Together with $\ell_{k+1} \le C \ell_k$
 this gives a lower estimate on $\Lambda_{\ell_k}$, as claimed in Lemma \ref{bourgain-lemma}.\\

\section{\bf Existence of polynomial patterns: Proof of Theorem \ref{180817thm1.1}}\label{181208section6}

Let $\mu$ be as in Theorem \ref{180817thm1.1}. Let $l_0$ be as in Theorem \ref{main-result}. Theorem \ref{180817thm1.1} follows if we are able to construct a Borel measure $\nu$ on $[0, 1]\times [0, 1]$ such that 
\beq\label{180817e7.1}
\nu([0, 1]\times [0, 1])>0
\endeq
and
\beq\label{180817e7.2}
\nu \text{ is supported on } X=\{(x, y)\in [0, 1]^2: x, y, x+P(y-x)\in E, \  2^{-l_0}\le y-x\le 2^{-l_0+1}\}. 
\endeq
For $\epsilon>0$, define 
\beq
\mu_{\epsilon}:=\mu * \vartheta_{\epsilon},
\endeq
where $\vartheta_{\epsilon}(x)=\epsilon^{-1}\vartheta(x/\epsilon)$. A standard argument shows that 
\beq
\mu_{\epsilon}\to \mu \text{ in } H^{-s_0} \text{ as } \epsilon\to 0.
\endeq
We define a linear functional $\nu$ acting on functions $f: [0, 1]^2\to \R$ by 
\beq\label{180817e7.4}
\langle \nu, f\rangle:=\lim_{\epsilon\to 0} \iint f(x, y) \mu_{\epsilon}(x+P(y-x)) \tau_{l_0}(y-x)d\mu_{\epsilon}(x) d\mu_{\epsilon}(y).
\endeq
The following lemma holds.
\begin{lem}\label{180817lemma7.1}
The limit in \eqref{180817e7.4} exists for every continuous function $f$. Moreover, 
\beq\label{180817e7.5}
|\langle \nu, f\rangle|\le C \|f\|_{\infty},
\endeq
where $C$ is independent of $f$. 
\end{lem}
\begin{proof}[Proof of Lemma \ref{180817lemma7.1}.]
For every $\epsilon>0$, the following inequality holds.
\beq
\begin{split}
& \Big|\iint f(x, y) \mu_{\epsilon}(x+P(y-x))\tau_{l_0}(y-x) d\mu_{\epsilon}(x) d\mu_{\epsilon}(y) \Big|\\
& \le \|f\|_{\infty} \iint \mu_{\epsilon}(x+P(y-x))\tau_{l_0}(y-x) d\mu_{\epsilon}(x) d\mu_{\epsilon}(y)
\end{split}
\endeq
By the Sobolev improving estimate in Proposition \ref{180816prop3.1}, this can be bounded by 
\beq
2^{\gamma_0 l_0} \|\mu_{\epsilon}\|^2_{H^{-s_0}} \|\mu_{\epsilon}\|_{H^{-s_0}}\le 2^{3+\gamma_0 l_0} \|\mu\|^3_{H^{-s_0}}.
\endeq
Recall that $1-s_0<\beta$. Under this assumption we know $\|\mu\|_{H^{-s_0}}$ is finite. This proves \eqref{180817e7.5} if the limit \eqref{180817e7.4} exists. \\

It remains to prove the existence of the limit \eqref{180817e7.4}. By density, it suffices to prove that the limit exists for every smooth function $f$ whose Fourier series consists of only finitely many terms. Hence it suffices to prove that the limit 
\beq
\lim_{\epsilon\to 0}\iint e^{iN_1 x+i N_2 y} \mu_{\epsilon}(x+P(y-x))\tau_{l_0}(y-x) d\mu_{\epsilon}(x) d\mu_{\epsilon}(y)
\endeq
exists for given $N_1, N_2\in \N$. By Proposition \ref{180816prop3.1}, 
\beq\label{180817e7.10}
\begin{split}
& \Big|\iint e^{iN_1 x+i N_2 y} (\mu_{\epsilon_1}-\mu_{\epsilon_2})(x+P(y-x))\tau_{l_0}(y-x) d\mu_{\epsilon}(x) d\mu_{\epsilon}(y)\Big|\\
& \le 2^{\gamma_0 l_0} \|\mu_{\epsilon_1}-\mu_{\epsilon_2}\|_{H^{-s_0}} \|\mu'\|_{H^{-s_0}}  \|\mu''\|_{H^{-s_0}},
\end{split}
\endeq
where 
\beq
d \mu'(x)= e^{iN_1 x} d\mu_{\epsilon}(x) \text{ and } d\mu''(x)=e^{iN_2 x}d\mu_{\epsilon}(x). 
\endeq
The right hand side of \eqref{180817e7.10} can be further bounded by 
\beq
C_{N_1, N_2} 2^{\gamma_0 l_0} \|\mu_{\epsilon_1}-\mu_{\epsilon_2}\|_{H^{-s_0}} \|\mu\|^2_{H^{-s_0}}.
\endeq
That the limit exists follows from the fact that $\|\mu_{\epsilon_1}-\mu_{\epsilon_2}\|_{H^{-s_0}}$ can be made arbitrarily small when $\epsilon_1$ and $\epsilon_2$ are chosen small enough. This finishes the proof of Lemma \ref{180817lemma7.1}. 
\end{proof}

After obtaining Lemma \ref{180817lemma7.1}, we apply the Riesz representation theorem and obtain a non-negative measure $\nu$ defined by \eqref{180817e7.4}. It remains to prove that $\nu$ satisfies the desired properties \eqref{180817e7.1} and \eqref{180817e7.2}. 

To prove \eqref{180817e7.1}, we write 
\beq
\begin{split}
\langle \nu, 1\rangle &=\lim_{\epsilon\to 0} \iint \mu_{\epsilon}(x+P(t)) \tau_{l_0}(t)d\mu_{\epsilon}(x) d\mu_{\epsilon}(x+t)\\
			&= \iint_{\R^2} \widehat{\mu}(\xi) \widehat{\mu}(\eta) \big[\int_{\R} e^{it\xi+iP(t)\eta}\tau_{l_0}(t)dt \big] \widehat{\mu}(\xi+\eta) d\xi d\eta. 
\end{split}
\endeq
From Theorem \ref{main-result}, it follows that $\langle \nu, 1\rangle \ge c_0>0$. This proves \eqref{180817e7.1}. 

Finally, we prove \eqref{180817e7.2}. Let us introduce 
\beq
\widetilde{X}:=\{(x, y)\in [0, 1]^2: x, y, x+P(y-x)\in E\}.
\endeq
By the definition of the measure $\nu$, it is enough to prove that $\nu$ is supported on $\widetilde{X}$. Let $f$ be a continuous function with $\text{supp}(f)$ disjoint from $\widetilde{X}$. We need to prove that $\langle \nu, f\rangle=0$. Since $E$ is closed, $\widetilde{X}$ is also closed. Moreover, $\text{dist}(\text{supp}(f), \widetilde{X})>0$. Using a partition of unity, we are able to write $f$ as a finite sum $\sum f_j$, where for each $j$, the function $f_j$ is continuous and satisfies at least one of the following 
\beq\label{180817e7.15}
\begin{split}
& \text{dist}(\text{supp}(f_j), E\times [0, 1])>0,\\
& \text{dist}(\text{supp}(f_j), [0, 1]\times E\})>0,\\
& \text{dist}\Big(\big\{x+P(y-x): (x, y)\in \text{supp}(f_j)\big\}, E\Big)>0.
\end{split}
\endeq
We will prove that $\langle \nu, f_j\rangle=0$ for every $j$. If $f_j$ satisfies either the first or the second condition in \eqref{180817e7.15}, then the integral in \eqref{180817e7.4} is $0$ for every $\epsilon$ small enough. If $f_j$ satisfies the third condition in \eqref{180817e7.15}, then the support of $f_j$ is a positive distance from the support of $\mu_{\epsilon}(x + P(y-x)$ for sufficiently small $\epsilon$, so the integral is again equal to $0$ if $\epsilon$ is sufficiently small. This finishes the proof of \eqref{180817e7.2}.

\end{document}